\documentclass[11pt]{amsart}

\usepackage{amsfonts}
\usepackage{amssymb}
\usepackage{url,graphicx,tabularx,array,geometry}
\usepackage{amsmath}
\usepackage{color}
\usepackage{epstopdf}
\usepackage{framed}
\usepackage{verbatim}
\usepackage{blkarray} 
\usepackage{enumitem} 
\usepackage{multirow} 
\usepackage[numbers]{natbib}
\usepackage{amsthm} 
\newtheorem{theorem}{Theorem}[section]
\newtheorem{lemma}[theorem]{Lemma}
\newtheorem{proposition}[theorem]{Proposition}
\newtheorem{corollary}[theorem]{Corollary}
\newtheorem{definition}[theorem]{Definition}

\theoremstyle{remark}
\newtheorem{remark}{Remark}

\newcommand{\R}{\mathbb{R}}
\newcommand{\N}{\mathbb{N}}
\newcommand{\Z}{\mathbb{Z}}

\DeclareMathOperator{\rank}{rank}
\DeclareMathOperator{\nrank}{rank_{+}}
\DeclareMathOperator{\xc}{xc}
\DeclareMathOperator{\stab}{STAB}

\DeclareMathOperator{\conv}{conv}

\newcommand{\MM}{{\mathcal M}}
\newcommand{\MMR}{{\mathcal R}}

\begin{document}

\title[Extension Complexity of Stable Set Polytopes for Perfect Graphs
]{On the Linear Extension Complexity of Stable Set Polytopes for Perfect Graphs
}
\thanks{This work is dedicated to the memory of Michel Deza, with gratitude for his stimulating role during the early career of the second author. }

\author{Hao Hu}
\address{Tilburg University, P.O. Box 90153, 5000 LE Tilburg, The Netherlands.}
\email{h.hu@tilburguniversity.edu}

\author{Monique Laurent}
\address{Centrum Wiskunde \& Informatica (CWI), Amsterdam and Tilburg University;
CWI, Science Park 123, Postbus 94079, 1090 GB Amsterdam, The Netherlands.}
\email{monique@cwi.nl}

\date{\today}

\maketitle


\begin{abstract}
We study the linear extension complexity of stable set polytopes of perfect graphs. 
We make use of known structural results permitting to decompose perfect graphs into basic perfect graphs by means of two graph operations: 2-join and skew partitions.  Exploiting the link between extension complexity and the nonnegative rank of an associated slack matrix, we investigate the behaviour of the extension complexity under these graph operations. We  show bounds for the extension complexity of the stable set polytope of a perfect graph $G$ depending linearly on the size of $G$   
and involving the depth of a decomposition tree of $G$ in terms of basic perfect graphs.

\end{abstract}

\section{Introduction} \label{CH_PP:introduction}

The polyhedral approach is a  classical, fundamental approach to solve   combinatorial optimization problems,  which aims to represent the convex hull of feasible solutions by linear inequalities and then to use linear programming to solve  the optimization problem. One of the major difficulties is that the explicit linear description of the corresponding polytope  may need an exponentially large number of linear inequalities (facets) in its natural space. This is  the case, e.g., for spanning tree polytopes or for matching polytopes while the corresponding combinatorial optimization problems are  in fact polynomially  solvable.
A widely investigated approach consists in searching for a  compact extension (aka extended formulation) of a given polytope $P$, i.e., searching for another polytope $Q$ lying in a higher dimensional space, which projects  onto $P$ and has less facets than $P$.
The smallest number of facets of such an extension is known as the {\em extension complexity} of $P$,  investigated in the seminal work of Yannakakis \cite{yannakakis1988expressing}.
The interest in this parameter lies in the fact that linear  optimization over $P$ amounts to linear optimization over $Q$.

Understanding which classes of polytopes have small extension complexity  has received considerable attention recently (see, e.g., the surveys \cite{CCZ10,Kaibel}). Well known classes admitting polynomial extension complexity include $\ell_1$-balls, spanning tree polytopes \cite{M,W}, permutahedra \cite{Goemans}. On the negative side Rothvoss \cite{Rothvoss12} showed the existence of $0/1$ polytopes whose extension complexity grows exponentially with the dimension and Fiorini et al. \cite{fiorini2015exponential}
show that this is the case for classes of combinatorial poytopes including cut, traveling salesman and stable set polytopes of graphs. In particular, a class of graphs on $n$ vertices is constructed in \cite{fiorini2015exponential} whose stable set polytope has  extension complexity at least $2^{\Omega(\sqrt n)}$.
While the latter polytopes correspond to hard cominatorial optimization problems, Rothvoss \cite{rothvoss2014matching} shows that  the  matching polytope of the complete graph $K_n$ has exponentially large extension complexity $2^{\Omega( n)}$, answering a long-standing open question of Yannankakis \cite{yannakakis1988expressing}.

 Yannakakis \cite{yannakakis1988expressing} investigated the extension complexity of  stable set polytopes for   perfect graphs. 
While  their linear inequality description is explicitly known (given by nonnegativity and clique constraints \cite{chvatal1985star}) it may involve exponentially many inequalites since perfect graphs may have exponentially many maximal cliques. This is the case, for instance, for  double-split graphs (see Section \ref{secdoublesplit} below). 
Using a reformulation of the extension complexity in terms of the nonnegative rank of the so-called  slack matrix and a link to communication complexity Yannakakis  \cite{yannakakis1988expressing} proved that the extension complexity for a perfect graph on $n$ vertices is  in the order $n^{O(\log n)}$.
It is an open problem whether this is the right regime or whether the extension complexity can be polynomially  bounded in terms of $n$.
This question is even more puzzling in view of the fact that compact {\em semidefinite} extensions (instead of linear ones) do exist. 
Indeed the 
 stable set polytope  of a perfect graph on $n$ vertices can be realized as projection of an affine section of the cone of $(n+1)\times (n+1)$ positive semidefinite matrices (using the so-called theta body, see \cite{GLS}). In fact  the only known polynomial-time algorithms for the maximum stable set problem in perfect graphs are based on semidefinite programming, and it is  open whether efficient  algorithms exist that are based on linear  programming.
We note that it has also been shown in \cite{LRS}  that compact semidefinite extensions do not exist for cut, traveling salesman and stable set polytopes for general graphs.

\medskip
In this paper we revisit the problem of finding upper bounds for the extension complexity of stable set polytopes for perfect graphs.
We make use of the recent decomposition results  for perfect graphs by Chudnovsky et al.\,\cite{chudnovsky2006strong},
who proved that any perfect graph can be decomposed into {\em basic perfect graphs} by means of two graph operations (special 2-joins and skew partitions). There are five classes of basic perfect graphs: bipartite graphs and their complements, line graphs of bipartite graphs and their complements, and double-split graphs.
As a second crucial ingredient we use the fundamental link established by Yannakakis \cite{yannakakis1988expressing} between the extension complexity of a polytope and the nonnegative rank of its slack matrix. We investigate how the nonnegative rank of the slack matrix behaves under the graph operation used for decomposing perfect graphs. This allows to upper bound the extension complexity of the stable set polytope of a perfect graph in terms of its number $n$ of vertices and the depth of a decomposition tree. As an application the extension complexity is polynomial for the class of perfect graphs admitting a decomposition tree whose depth is logarithmic in $n$.

\medskip
The paper is organized as follows. In Section \ref{secpreliminaries}  we recall definitions and preliminary results that we need in the paper. First we consider extended formulations and  slack matrices  and we recall the result of Yannakakis \cite{yannakakis1988expressing} expressing the extension complexity in terms of the nonnegative rank of  the slack matrix. 
After that we consider perfect graphs and their stable set polytopes and recall the structural decomposition result of  \citet{chudnovsky2006strong} for  perfect graphs. 
In Section \ref{secbasic} we consider the basic perfect graphs and show that their extension complexity is (at most)  linear in the number of vertices and edges.
In Section \ref{secoperations} we consider the behaviour of the extension complexity of the stable set polytope under several graph operations: graph substitution, 2-joins and skew partitions. Finally in  Section \ref{secappli} we  use these results to upper bound the extension complexity for arbitrary perfect graphs in terms of the number of vertices and edges and of the depth of a decomposition tree into basic perfect graphs.


\section{Preliminaries}\label{secpreliminaries}

Here we group some definitions and preliminary results that we will need in the rest of the paper.
In Section \ref{secextendedformulation} we consider extended formulations of polytopes and recall the fundamental result of Yannakakis \cite{yannakakis1988expressing} which characterizes the extension complexity of a polytope in terms of the nonnegative rank of its slack matrix. Then in
 Section \ref{secperfectgraph} we recall  results about perfect graphs and their stable set polytopes.
 
Throughout we use the following  notation. 
We let $\conv(V)$ denote the convex hull of a set $V\subseteq \R^d$. For an integer $n\in \N$, we set $[n]=\{1,\cdots,n\}$. Given a subset $S\subseteq [n]$, $\chi^S\in\{0,1\}^n$ denotes its characteristic vector. 
 Given a graph $G=(V,E)$ and a subset $S\subseteq V$, $G[S]$ denotes the subgraph of $G$ induced by $S$, with vertex set $S$ and edges all pairs $\{u,v\}\in E$ with $u,v\in S$.
 
\subsection{Extended formulation, extension complexity and slack matrix} \label{secextendedformulation}
An extended formulation of a polytope is a linear system describing this polytope possibly using additional variables. The interest of extended formulations is due to the fact that one can sometimes reduce the number of inequalities needed to define the polytope when additional variables are  allowed.

\begin{definition}[Extended formulation]
	Let $P\subseteq \R^{d}$ be a polytope. The linear system 
	\begin{equation}
	Ex + Ft = g,  \; \hat{E}x + \hat{F}t \leq \hat{g}, \label{EF}
	\end{equation}
	in the variables $(x,t) \in \R^{d}\times\R^q$, is called an {\em extended formulation} of $P$  if the following equality holds:
	\begin{equation*}\label{eqEF}
	P=\{x\in\R^d: \exists t\in \R^q \text{ s.t. } Ex + Ft = g,  \; \hat{E}x + \hat{F}t \leq \hat{g}\}.
	\end{equation*}	
Here, the matrices $E, \hat E$ have $d$ columns, the matrices $F, \hat F$ have $q$ columns,	 the additional variable $t$ is called the {\em lifting variable}, and the {\em size} of the extended formulation is defined as the number of inequalities in the system (\ref{EF}) (i.e., the number of rows of the matrices $\hat E,\hat F$).
The extended formulation is said to be in {\em slack form} if the only inequalities are nonnegativity conditions on the lifting variable $t$, i.e., if it is of the form:
	\begin{equation}
	Ex + Ft = g, \; t \geq 0 \label{EFslack}
	\end{equation}
	and then its size is the dimension of the variable $t$.
\end{definition}

\begin{remark}\label{EFlemma}
The linear system (\ref{EF}) is an extended formulation of $P$ if and only if  for every vertex $v$ of $P$ there exists a lifting variable $t_v$ such that the vector $(v,t_v)$ satisfies (\ref{EF}). 
\end{remark}


\begin{definition}[Extension complexity]
	Let $P\subseteq \R^{d}$ be a polytope. A polytope $Q \subseteq \R^{k}$ is called an {\em extension} of $P$ if there exists a linear mapping $\pi: \R^{k} \rightarrow \R^{d}$ such that  $P=\pi(Q)$. The {\em size} of the extension $Q$, denoted by size$(Q)$, is defined as the number of facets of $Q$. Then the  {\em extension complexity} of $P$ is the parameter $\xc(P)$ defined as 
	$$\xc(P) =\min \{ \text{\rm size}(Q) : Q \text{ is an extension of } P  \}.$$\end{definition}


As we recall in Theorem \ref{FactorizationTheorem} below,
 extended formulations and extensions are in fact equivalent notions and the extension complexity of $P$ can be computed via the nonnegative rank of its slack matrix.  
 
\begin{definition}[Slack matrix] \label{DEF:Slackmatrix}
	Let $P \subseteq \R^{d}$ be a polytope. Consider a linear system $Ax\leq b$ describing $P$, i.e.,  $P = \{ x \in \R^{d} : Ax\leq b \}$ with $A\in \R^{m \times d}$ and $b \in \R^{m}$, and a set  $V = \{ v_{1}, \ldots ,v_{n} \} \subseteq P$ containing all the vertices of $P$, i.e., $P=\conv(V)$. Then the $m \times n $   matrix $S = (S_{i,j})$ with entries 
	$$S_{i,j} = b_{i} - A_{i}^{T}v_{j} \ \ \text{ for } i\in [m],\  j\in [n]$$ 
	is called a {\em slack matrix} of $P$, said to be   induced by $V$ and the linear system $Ax\leq b$. 
\end{definition}

Note that choosing different point sets and linear systems  in Definition $\ref{DEF:Slackmatrix}$ will induce  different slack matrices. However,  it will follow from  Theorem \ref{FactorizationTheorem} below that  they all    have the same nonnegative rank. 
So we may speak of {\em the}  slack matrix of $P$ without referring explicitly to the selected point set and linear system if there is no ambiguity.

\begin{definition}[Nonnegative rank]
	The {\em nonnegative rank} of a nonnegative matrix $S \in  \R_{+}^{m \times n}$ is defined as
	$$\nrank(S) = \min \{ r : \exists \; T \in  \R_{+}^{m \times r}\; \; \exists \; U \in  \R_{+}^{r \times n}  \text{ such that } S= TU \}.$$
	In what follows, a decomposition of the form $S= TU$ as above  is called a {\em nonnegative decomposition with intermediate dimension $r$}.
\end{definition}

We  refer, e.g.,  to \cite{Gillis} for an overview of applications of the nonnegative rank and for further references.
The following are easy  well-known properties of the nonnegative rank, which we will extensively use later.

\begin{lemma} \label{ranklemmaBlock}
	\begin{enumerate}[topsep=0pt,itemsep=-1ex,partopsep=1ex,parsep=2ex,label=(\roman*)]
		\item For  $S\in \R_+^{m\times n}$, we have \ $\nrank (S)  = \nrank (S^{T})\le \min\{m,n\}$.
		\item For   $S_1\in \R_+^{m\times n_1},$  $S_2\in\R_+^{m\times n_2}$, we have \ $\nrank (S_1 \ S_2) 
		 \leq \nrank(S_{1}) + \nrank(S_{2})$.
		
	\end{enumerate}
\end{lemma}

We can now formulate  the following  result of Yannakakis \cite{yannakakis1988expressing}, which  establishes a fundamental link between extended formulations, the extension complexity of a polytope and the nonnegative rank of its slack matrix. We also refer, e.g., to \cite{GPT} for a detailed exposition in the more general setting of conic factorizations.

\begin{theorem} \emph{\cite{yannakakis1988expressing}} \label{FactorizationTheorem}
	Let $P = \{x \in \R^{d} : Ax \leq b \}$ be a polytope whose dimension is at least one, where $A \in \R^{m \times d}$ and $b \in \R^{m}$,  let $V = \{ v_{1},\ldots,v_{n} \}$ be a subset of $P$ containing the set  of vertices of $P$, and let $S \in \R^{m \times n}$ be the induced slack matrix. Let $r$ be a positive integer. The following assertions are equivalent:
	\begin{enumerate} 
		\item[(i)] $\nrank(S) \leq r$;
		\item[(ii)]  $P$ has an extension of size at most $r$;
		\item[(iii)]  $P$ has an extended formulation in slack form of size at most $r$;
		\item[(iv)]  $P$ has an extended formulation of size at most $r$. 
	\end{enumerate}
\end{theorem}

Hence the extension complexity of $P$ can be defined by any of the following formulas:
$$\begin{array}{ll}
\xc(P) & =\  \min\{r: P \text{ has an extension of size } r\}\\
& =\  \min\{r: P \text{ has an extended formulation (in slack form) of size } r\}\\
&=\  \rank_+(S) \ \text{  for any slack matrix of } P.
\end{array}
$$
In particular the extension complexity of a $d$-dimensional polytope is at least $d+1$.

\subsection{Stable set polytopes  and perfect graphs}\label{secperfectgraph}

Given a graph $G=(V,E)$, a {\em stable set} of $G$ is a subset $I\subseteq V$ where no two elements of $I$ form an edge of $G$.  
The maximum cardinality of a stable set in $G$ is the stability number of $G$, denoted by $\alpha(G)$.   The {\em stable set polytope} $\stab(G)$ of  $G$ is defined as the convex hull of the characteristic vectors of the stable sets of $G$:
\[ \stab(G):= \conv \{ \chi^{I} : I \text{ is stable in }  G \}  \subseteq \R^{ |V|  }. \]

Computing the stability number $\alpha(G)$ is an NP-hard problem and accordingly the full linear  inequality description of  the stable set polytope is not known in general. 
However, for some classes of graphs, there exist efficient algorithms for computing $\alpha(G)$ and an explicit linear inequality description of $\stab(G)$ is known. 
This is the case in particular for the class of perfect graphs, as we now recall.

\newcommand{\oG}{\overline G}

The {\em chromatic number} $\chi(G)$ is the minimum number of colours that are needed to properly color the vertices of $G$, in such a way that two adjacent nodes receive distinct colors. The clique number of $G$ is the largest cardinality of a clique in $G$, denoted by $\omega(G)$. Clearly, $\chi(G)\ge \omega(G)$. 
Following Berge \cite{Berge} a graph $G$ is said to be {\em perfect} if $\chi(G')=\omega(G')$ for each induced subgraph $G'$ of $G$.
A classical result of  Lov\'asz \cite{Lo72a,Lo72b} shows that  $G$ is perfect if and only if its complement $\oG$ is perfect.

Going back to the stable set polytope of $G$, it is clear that for any  clique $C$ of $G$, the following linear inequality $\sum_{v\in C}x_v\le 1$ (called a {\em clique inequality}) is valid for the stable set polytope.
An early result of 
Chv\'atal  \cite{chvatal1975certain} shows that  perfect graphs can be characterized as those graphs for which the clique inequalities together with nonnegativity fully describe the stable set polytope.

\begin{theorem}\emph{\cite{chvatal1975certain}} \label{perfectSTABtheorem}
	A  graph $G=(V,E)$ is perfect if and only if $\stab(G)$ is characterized by the following linear system, in the  variables $x \in \R^{|V|}$:
	\begin{align}  \label{perfectSTABtheorem_equation1}
	\sum_{v\in C} x_{v} & \leq 1& \; \; \forall \; C \text{ maximal clique of } G,  \\
	x_v & \geq 0 & \; \; \forall \;  v \in V.  \label{perfectSTABtheorem_equation2}
	\end{align}
\end{theorem} 

Hence, when $G$ is a perfect graph, its stable set polytope   $\stab(G)$ can be characterized by the nonnegativity constraints and the maximal clique constraints,
as stated in Theorem  \ref{perfectSTABtheorem}. However,  the number of maximal cliques of $G$ might be exponentially large, and thus this result does not lead directly to  an efficient algorithm for solving the maximum stable set problem in perfect graphs.
As a matter of fact, as of today, the only known efficient algorithm for this problem is based on using semidefinite programming, as shown by Gr\"otschel, Lov\'asz and Schrijver \cite{GLS}.
It is not known whether an efficient linear programming based algorithm exists for solving this problem. This motivates  our work in this paper to investigate the extension complexity of the stable set polytope of perfect graphs. 

\newcommand{\MI}{{\mathcal I}}
\newcommand{\MC}{{\mathcal C}}

Throughout we use the following notation: for a graph $G$, $\MI_G$ denotes the set of stable sets of $G$ and  $\MC_G$ denotes the set of {\em maximal} cliques of $G$.
We will use the following slack matrix for the stable set polytope of perfect graphs.

\begin{definition}\label{slackG}
Given a  graph $G=(V,E)$,  $S_G$ denotes the slack matrix of $\stab(G)$, whose rows are indexed by $V \cup\MC_G$ (corresponding to the nonnegativity constraints (\ref{perfectSTABtheorem_equation2})  and 
the maximal clique constraints (\ref{perfectSTABtheorem_equation1})), and whose columns are indexed by $\MI_G$, with entries
$$\begin{array}{ll}
S_G(v,I) =|\{v\}\cap I|,\ \  S_G(C,I) = 1-|I\cap C| & \text{ for } v\in V, \ C\in \MC_G,\  I\in \MI_G.
\end{array}$$
\end{definition}

From Theorem \ref{FactorizationTheorem}, we know that $\nrank(S_{G}) = \xc(\stab(G))$ when $G$ is perfect.
Hence to study the extension complexity of $\stab(G)$ we need to gain insight on the nonnegative rank of the slack matrix and for this we will use the  structural decomposition result for perfect graphs from 
 \cite{chudnovsky2006strong,thesis,JGT},  that we recall below.


 Berge \cite{Berge} observed that  if a graph $G$ is perfect  then neither $G$ nor $\oG$ contains an  induced cycle of odd length at least $5$,  and he asked whether the converse is true.
This was answered in the affirmative  recently by \citet{chudnovsky2006strong},
a result known as the  {\em strong perfect graph theorem}.
 The  proof of this result in \cite{chudnovsky2006strong} relies on a structural decomposition result for perfect graphs. We need some  definitions to be able to state  this decomposition result.  

\medskip
First we introduce double-split graphs, which form an additional class of basic graphs considered in \cite{chudnovsky2006strong}, next to bipartite graphs, line graphs of bipartite graphs and their complements.

\begin{definition} \emph{\cite{chudnovsky2006strong}}\label{defsplit}
Consider integers $p,q\ge 2$ and sets $L_1,\ldots,L_p\subseteq [q]$. 	A graph $G=(V,E)$ is  a {\em double-split graph}, with parameters $(p,q,L_1,\ldots,L_p)$, if $V$  can be partitioned as $V=V_1\cup V_2$, where  $V_{1}=\{a_1,b_1,\ldots,a_p,b_p\}$,  $V_{2}=\{x_1,y_1,\ldots,x_q,y_q\}$ and
	\begin{enumerate}[topsep=0pt,itemsep=-1ex,partopsep=1ex,parsep=1ex,label=(\roman*)]
		\item[$\bullet$] $G[V_{1}]=(V_1,E_1)$ is a disjoint union of edges, \ 
				 $G[V_{2}]=(V_2,E_2)$ is the complement of a disjoint union of edges,  say
		$$E_{1}  = \{ \{a_{i},b_{i}\} :  i \in [p] \},\ \ E_2=\{\{x_i,y_j\}: i\ne j\in [q]\};$$
		\item[$\bullet$] 
		 The only edges between $V_1$ and $V_2$ are the pairs 
		$\{a_i,x_j\}, \{b_i,y_j\} \text{ for } i\in[p], j\in L_i$, and the pairs
		$\{a_i,y_j\}, \{b_i,x_j\}\ \text{ for } i\in[p], j\in [q]\setminus L_i.$
		
	\end{enumerate}
\end{definition}
Note that double-split graphs may have at the same time exponentially many maximal cliques and exponentially many maximal independent sets (when choosing, e.g., $p=q$).

The decomposition result for perfect graphs needs two graph operations: $2$-joins and skew partitions.

\begin{definition}{\rm \cite{CC}} 
\label{def2join}
	A {2-join} of $G=(V,E)$ is a partition of $V$ into $(V_1,V_2)$  together with  disjoint nonempty subsets $A_{k},B_{k} \subseteq V_{k}$ (for $k = 1,2$) such that 
every vertex of $A_1$ (resp., $B_1$) is adjacent to every vertex of $A_2$ (resp., $B_2$) and there 
are no other edges between $V_1$ and $V_2$.
\end{definition}

\begin{definition}{\rm \cite{chvatal1985star,chudnovsky2006strong}}\label{skew}
	A {\em skew partition} of  $G=(V,E)$ is a partition of $V$ into four nonempty sets $(A_{1},B_{1},A_{2},B_{2})$  such that every vertex in $A_1$ is adjacent to every vertex in $A_2$, and there are  no edges between  vertices in $B_{1}$ and  vertices in $B_{2}$. 

\end{definition}

The following decomposition  result for perfect graphs involves 2-joins and skew partitions with refined properties, namely {\em proper} 2-joins and {\em balanced} skew partitions. As these additional properties will play no role in our treatment, we do not include the exact definitions.

\begin{theorem}{\rm \cite[Statement 1.4]{chudnovsky2006strong}}
	\label{perfectDecomposition}
	Let $G$ be a  perfect graph. Then,   either $G$ belongs to one of the following five basic classes: bipartite graphs and their complements, line graphs of bipartite graphs and their complements, double-split graphs; or one of $G$ or $ \bar{G}$ admits a proper 2-join;
	or 	$G$ admits a balanced skew partition.
\end{theorem}

In this paper we investigate how the extension complexity of the stable set polytope of a perfect graph $G$ can be upper bounded depending on the two decomposition operations (2-joins and skew partitions) that are needed to build $G$ from the basic graph classes.

\section{Extension complexity for basic perfect graphs}\label{secbasic}

In this section we show bounds for the extension complexity of the stable set polytope for the basic classes of perfect graphs.
Recall the definition of the slack matrix $S_G$ introduced in Definition \ref{slackG}.
From Theorem \ref{FactorizationTheorem}, we know that when $G$ is  perfect,  the extension complexity of its stable set polytope  is given by the nonnegative rank of the matrix $S_G$:
$$ \xc(\stab(G))=\nrank(S_{G}).$$
So in order to upper bound $\xc(\stab(G))$ it suffices to upper bound $\nrank(S_G)$.
The following  upper bound follows directly from Lemma \ref{ranklemmaBlock}(i) (since 
$S_G$ has $|V(G)|+|\MC|$ rows).

\begin{lemma}\emph{\cite{yannakakis1988expressing}} \label{CH_PG:SGlemma1}
	Let $G= (V,E)$ be a perfect graph and let $\MC$ denote its set of maximal cliques. Then we have: $\xc(\stab(G)) \leq |V| + |\MC|$.
\end{lemma}

As an example of application of Lemma \ref{CH_PG:SGlemma1}, $\xc(\stab(G))\le 2|V|$ when $G$ is a chordal graph (i.e., has no induced cycle of length at least 4, since then $G$ has at most $|V|$ maximal cliques). Moreover,  for the  complete graph $K_p$,  $\xc(\stab({K_p}))\le p+1$, since $K_p$ has a unique maximal clique. As  $\stab(K_p)$ has dimension $p$, the reverse inequality holds and thus $\xc(\stab(K_p))=p+1$. 
For the complement of  $K_p$,  $\xc(\stab(\overline{K_p}))\le 2p$, since $\overline{K_p}$ has $p$ maximal cliques.
In fact as $\stab(\overline{K_p})=[0,1]^p$ we have
$\xc(\stab(\overline{K_p}))=2p$   \cite{FKPT}.

 Furthermore, 
using   Lemma \ref{ranklemmaBlock}(ii), one can verify that when $G$ is perfect  the extension complexity of the  stable set polytope  of  $G$ and its complement $\oG$ are linearly related.

\begin{lemma}\emph{\cite{yannakakis1988expressing}} \label{lemGbar}
	Let $G= (V,E)$ be a perfect graph  and let $\oG$ be its complement. Then  \[\xc(\stab(\oG)) \leq \xc(\stab(G)) + |V|.\]
\end{lemma}

In what follows we first present a simple  bounding technique for the extension complexity which we then apply to double-split graphs.
After that we consider the extension complexity for the other four basic classes of perfect graphs.

\subsection{A simple bounding technique and double-split graphs}\label{secdoublesplit}

We begin with a simple bounding technique based on considering a partition $V=V_1\cup V_2$ of  the vertex set of $G=(V,E)$.

 Below and later in the paper we will use the following notation. 
For $k=1,2$ we let $G_k=G[V_k]$ denote the sugraph of $G$ induced by $V_k$,
$\MC_k$ denotes the set of maximal cliques of $G_k$ and 
$\MI_k$ denotes the set of independent sets of $G_k$. 
In addition $\MC_{12}$ (resp., $\MI_{12}$) denotes the set of `mixed' maximal cliques (resp.,   `mixed'  independent sets) of $G$, i.e., those that meet both $V_1$ and $V_2$.
Finally, we set $\MMR_k=V_k\cup \MC_k$, so that  the rows of the slack matrix $S_G$ of the stable polytope of $G$ are indexed by the set $\MMR_1\cup\MMR_2\cup \MC_{12}$, while  the columns of $S_G$ are indexed by the set $\MI_1\cup\MI_2\cup\MI_{12}$. With respect to these partitions of its row and column index sets, the matrix  $S_G$
has the following block form:


\begin{equation}\label{eqSG}
S_{G}  = 
\bordermatrix{ 		& \mathcal{I}_{1} & \mathcal{I}_{2} & \mathcal{I}_{12} \cr
\mathcal{R}_{1}	& S_{1,1} &	S_{1,2}	  &	S_{1,3}	\cr
\mathcal{R}_{2} & S_{2,1} & S_{2,2}   & S_{2,3} \cr
\mathcal{C}_{12}    & S_{3,1} &	S_{3,2}	  &	S_{3,3}	\cr  }
\end{equation}

\begin{lemma}\label{lemsplit}
Let $G=(V,E)$ be a perfect graph and let $V=V_1\cup V_2$ be a partition of its vertex set. Then we have
$$\xc(\stab(G))\le \xc(\stab(G[V_1]))+\xc(\stab(G[V_2])) +|\MC_{12}|,$$
where $\MC_{12}$ denotes the set of maximal cliques of $G$ that meet both $V_1$ and $V_2$.
\end{lemma}

\begin{proof}
We use the form of the slack matrix $S_G$ in (\ref{eqSG}). By construction, for $k=1,2$, 
we have $S_{k,k}=S_{G_k}$, 
 each column of $S_{k,3}$ is the copy of a column of $S_{k,k}$, and each column of $S_{1,2}$ (resp., $S_{2,1}$) coincides with the column of $S_{1,1}$ (resp., $S_{2,2}$) indexed by the empty set. 
Hence 
$ \nrank(S_{k,1}\ S_{k,2}\ S_{k,3}) = \nrank(S_{G_{k}})$ holds for $k=1,2$.
	Finally, we have
	$\nrank(
	S_{3,1} \ S_{3,2} \ S_{3,3})
	 \leq |\MC_{12}|$ since this matrix has  $|\MC_{12}|$ rows.
	Combining these and applying Lemma \ref{ranklemmaBlock} to $S_{G}$, we obtain the desired inequality.
\end{proof}

As an application, if $G$ is the disjoint union of two graphs $G_1$ and $G_2$ then we have the inequality:
$\xc(\stab(G))\le \xc(\stab(G_1))+\stab(G_2))$. As another application, we can upper bound the extension complexity for double-split graphs.

\begin{lemma}\label{lemdoublesplit}
	If $G=(V,E)$ is a double-split graph with parameters $(p,q,L_1,\cdots,L_p)$ then 
	\[\xc(\stab(G)) \leq  5p + 5q \le 5|V|/2\ \text{ and } \xc(\stab(\oG))\le 5p+5q\le  5|V|/2. \]
	\end{lemma}

\begin{proof}
The inequality $5p+5q\le 5|V|/2$ is clear since $|V|=2p+2q$. 
As $G$ is perfect and its complement is again a double-split graph (exchanging $p$ and $q$), it  suffices to show the inequality $\nrank(S_G) \leq  5p + 5q$. 
For this we use Lemma~\ref{lemsplit}, with the partition $V=V_1\cup V_2$ in the definition of a double-split graph from Definition~\ref{defsplit}. As $G_1=G[V_1]$ is a disjoint union of $p$ edges,  $G_1$ has  $p$ maximal cliques and $2p$ vertices, which  implies  $\nrank (S_{G_1})\le 3p$ (by Lemma \ref{CH_PG:SGlemma1}). As  $G_2=G[V_2]$ is the complement of the disjoint union of $q$ edges, we obtain  $\nrank(S_{G_2}) \le 3q+2q=5q$ (using Lemma \ref{lemGbar}). 
Finally there are $2p$  maximal cliques in $\MC_{12}$, given by the sets 
$\{a_{i}\}  \cup x_{L_{i}} \cup  y_{\overline{L_{i}}}, \text{ and }  \{b_{i}\}  \cup x_{\overline{L_{i}}} \cup  y_{L_{i}} \text{ for }   i \in [p].$
Hence, applying Lemma \ref{lemsplit} we obtain that $\xc(\stab(G))\le \xc(\stab(G_1))+\xc(\stab(G_2)) +|\MC_{12}|$ is upper bounded by  $\nrank(S_{G_1})+\nrank(S_{G_2})+|\MC_{12}|\le 3p+5q+2p=5p+5q.$
\end{proof}

\subsection{Bipartite graphs and  their line graphs and complements }

We just saw in Lemma \ref{lemdoublesplit} that the extension complexity of the stable set polytope of double-split graphs is linear in $|V|$. We now consider the other classes of basic perfect graphs.

The next bound for  bipartite graphs and their complements is well known and follows directly from Lemma \ref{CH_PG:SGlemma1} combined with Lemma \ref{lemGbar}. 

\begin{lemma}\label{SGlemma3} 
Let  $G= (V,E)$ be  a bipartite graph. Then 
	\[\xc(\stab(G)) \leq |V| + |E| \text{ and } \xc(\stab(\oG)) \leq 2 |V| + |E|.\]
\end{lemma}

Recently Aprile et al.  \cite{AFFHM} showed the following alternative upper bound for bipartite graphs:
$\xc(\stab(G))=O(|V|^2/\log |V|)$, which is thus sharper than the bound $|V|+|E|$ when the number of edges is quadratic in $|V|$. 
Moreover   a class of bipartite graphs $G$ is constructed in \cite{AFFHM} for which $\xc(\stab(G))=\Omega(|V|\log |V|)$.
Finding the exact regime of the extension complexity for bipartite graphs is still open.

Next we see that for  line graphs of bipartite graphs and their complements, the extension complexity is linear in $|V|$.

\begin{lemma}\label{SGlemma5}
Let  $G= (V,E)$ be  the line graph of a bipartite graph. Then 
	\[\xc(\stab(G)) \leq 2 |V| \text{ and } \xc(\stab(\oG)) \leq 3  |V|.\]
\end{lemma}

\begin{proof}
Assume  $G$ is the line graph of a bipartite graph $G'=(V',E')$. Then  $V(G)=E'$ and $\stab(G)$ is the matching polytope $M(G')$ of $G'$.
For $v\in V'$, let $\delta(v)$ denote the set of edges in $G'$ incident to $v$, called the star of $v$, and let $W$ be the set of vertices $v\in V'$ for which $\delta(v)$ is maximal (i.e.,  not strictly contained in the star of another vertex of $G'$). Then $\stab(G)=M(G')$ is defined by the nonnegativity constraints $x_e\ge 0$ ($e\in E'$) and the star constraints $\sum_{e\in \delta(v)} x_e\le 1$ for $v\in W$. We  show that in the description of $M(G')$ we need to consider at most $|E'|$ star constraints.  Clearly we may assume that $G'$ is connected (else consider each connected component).
If some node $v\in W$  is adjacent to a unique other node $u\in V'$ then $G'$ consists only of  the edge $\{u,v\}$ and it is clear that one star constraint  suffices.
Otherwise we may assume that each node $v\in W$ has degree at least 2, which implies $|E'|\ge |W|$ and thus the number of star constraints is at most $|E'|$.
Summarizing, the matching polytope of $G'$ is defined by at most $2|E'|$ linear constraints, which shows that $\stab(G)$ is defined by at most $2|E'|=2|V|$ linear constraints.
\end{proof}

Finally we show an upper bound which is uniform for all basic perfect graphs, which we will use in Section \ref{secappli}  to deal with general perfect graphs.
\begin{corollary}\label{corbasic}
For every basic perfect graph $G=(V,E)$, 
$\xc(\stab(G))\le 2(|V|+|E|)$ holds.
\end{corollary}

\begin{proof}
The claim is obvious if $G$ is bipartite or the line graph of a bipartite graph.
Assume now  $G=(V,E)$ is bipartite and $\overline G=(V,\overline E)$ is not bipartite (thus $n\ge 3$). It suffices to show 
$|E| \le 2|\overline E|$, which then implies $\xc(\overline G)\le 2|V|+|E|\le 2(|V|+|\overline E|)$.
As $|\overline E|={|V|\choose 2}-|E|$,  $|E| \le 2|\overline E|$ is equivalent to $|E|\le {|V|(|V|-1)/3}$, which follows from $|E|\le |V|^2/4$.

Consider now the case when $G$ is the line graph of a bipartite graph $G'$. By Lemma \ref{SGlemma5},
$\xc(\stab(\overline G))\le 3|V|$. We show that $\xc(\stab(\overline G))\le 2(|V|+|\overline E|)$, 
which follows if we can show   $|V|\le 2|\overline E|$.
If $\overline G$ has no isolated vertex then $|V|\le 2|\overline E|$ indeed holds. Assume now  $\overline G$ has an isolated vertex. Then $G$ has a vertex adjacent to all other vertices, which means $G'$ has an edge incident to all other edges of $G'$.
This implies that $G$ is the union of two cliques intersecting at a single vertex and thus $\overline G$ is a bipartite graph, so we are done as this case was treated above.

Finally if $G$ is a double-split graph, then it has no isolated vertex and thus, by Lemma~\ref{lemdoublesplit},
$\xc(\stab(G))\le 5|V|/2\le 2(|V|+|E|)$.
\end{proof}

\section{Graph operations}\label{secoperations}

We now consider some graph operations that play an important role when dealing with perfect graphs.
The operation of ``graph substitution" was first considered by Lov\'asz \cite{Lo72a} as crucial tool for his  {\em perfect graph theorem}, stating that the class of perfect graphs is closed under taking graph complements. 
After that we consider the two graph operations: 2-joins and skew partitions, that are used in the structural characterization of perfect graphs by \citet{chudnovsky2006strong}.

\subsection{Graph substitution}\label{graphsubstitution}
In this section, we consider the behaviour of the extension complexity of $\stab(G)$ when $G$ is obtained  from  two other graphs $G_1$ and $G_2$ via the  ``graph substitution" operation. This operation preserves perfect graphs: if $G_1$ and $G_2$ are perfect, then $G$ is perfect  \cite{chvatal1975certain}.

\begin{definition}
	Let $G_{1}=(V_{1},E_{1})$ and $G_{2}=(V_{2},E_{2})$ be two vertex-disjoint graphs and let $u$ be a vertex of $G_{1}$. {\em Substituting $G_2$ in $G_1$ at $u$} 
	produces the graph $G = \mathcal{S}(G_{1},u,G_{2})$, where  $G = (V,E)$  with
	$$ V = (V_{1} \backslash \{u\}) \cup V_{2},$$ 
	$$ E =   E(G_{1}[V_{1}\backslash \{u\}]) \cup E_{2} \cup \bigcup_{v\in V_2}\big\{\{v,w\}: \{u,w\}\in E_1\big\}.$$
\end{definition}

We  show that the extension complexity of $\stab(G)$ is bounded by the sum of the extension complexities of $\stab(G_{1})$ and $\stab(G_{2})$. We will use the following lemma.

\begin{lemma}  \label{sublemma1}
	Let $P$ be a nonempty polytope. Consider an extended formulation of $P$:
	\begin{equation}
	Ex+Fs=g, \; s \geq 0. \label{sublemma1_eq}
	\end{equation}
	 If the pair $(x_{0},s_{0})$ satisfies $Ex_{0}+Fs_{0}=0$ and $s_{0} \geq 0$,  then $x_{0} = 0$.
\end{lemma}
\begin{proof}
As $P\neq \emptyset$ there exists a feasible solution 	$(x,s)$  of  $(\ref{sublemma1_eq})$. For any $\lambda \ge 0$,  $(x,s)+ \lambda  (x_{0},s_{0})$ also satisfies  $(\ref{sublemma1_eq})$, which implies $x + \lambda  x_{0} \in P$ and thus $x_0=0$ since  $P$ is  bounded.\end{proof}


\newcommand{\ou}{\bar{u}}

\begin{theorem} \label{subtheorem}
	Let $G_{1}=(V_{1},E_{1})$ and  $G_{2}=(V_{2},E_{2})$ be  vertex-disjoint graphs and $u\in V_1$.
	 If $G=\mathcal{S}(G_{1},u,G_{2})$ is the graph obtained by substituting $G_{2}$ in $G_{1}$ at $u$, then we have 
	$$\xc(\stab(G)) \leq \xc(\stab(G_{1})) + \xc(\stab(G_{2})).$$
\end{theorem}

\begin{proof}
We use the following notation: for $x\in \R^{|V|}$ and $S\subseteq V$, $x(S)=(x_v)_{v\in S}$ denotes the restriction of $x$ to its entries indexed by $S$. 
For $i = 1,2$, 
set $r_i=\xc(\stab(G_i))$ and assume the linear system
	\begin{equation}
	E_{i}x_{i}+F_{i}s_{i}=g_{i}, \; s_{i} \geq 0 \label{sublemma2_sys1}
	\end{equation}
	is an extended formulation in slack form of $\stab(G_{i})$ of size $r_{i}$ (whose existence  follows from  Theorem \ref{FactorizationTheorem}), with  variables $x_i\in \R^{|V_i|}$ and lifting variables $s_{i} \in \R^{r_{i}}$.
	
Consider now variables 	$y_i\in \R^{|V_i|}$ and $t_i\in \R^{r_i}$ for $i=1,2$.
For convenience set $y_1(\ou)=y_1(V_1\setminus\{u\})$, so that  $y_1=(y_1(\ou),y_1(u))$ and the vector
$(y_1(\ou),y_2)$ is indexed by the vertex set $V$ of $G$.
 We claim that the  linear system 
	\begin{equation}
		\begin{cases}
	E_{1}y_{1}+F_{1}t_{1} = g_{1}, \; & t_{1}\geq 0, \\ 
	E_{2}y_{2}+F_{2}t_{2} - g_{2} \cdot y_{1}(u) = 0, \; & t_{2}\geq 0
	\end{cases} \label{sublemma2sys3}
	\end{equation}
provides an extended formulation of $\stab(G)$, with lifting variables $(t_1,t_2,y_1(u))$. As its size is equal to $r_1+r_2$ this implies the desired inequality $\xc(\stab(G))\le r_1+r_2$.
	
To prove that (\ref{sublemma2sys3})  is an extended formulation of $\stab(G)$,  we have to show that a vector $(y_{1}(\bar{u}), y_{2})$ belongs to $\stab(G) $ if and only if there exists a vector $(t_{1},t_{2},y_{1}(u))$ in $\R^{r_{1}+r_{2}+1}$ for which the vector $(y_{1},y_{2},t_{1},t_{2})$ 
satisfies the linear system  (\ref{sublemma2sys3}), where we set $y_1=(y_1(\ou),y_1(u))$.
	
We first show the 	``only if'' part.
 In view of Remark \ref{EFlemma} we may assume that  $(y_{1}(\bar{u}), y_{2})$ is a vertex of $\stab(G)$. Then $(y_{1}(\bar{u}), y_{2})$ is the characteristic vector $\chi^I$ of a stable set $I$ in $G$. Then the set $I_1=I\cap V_1$ is a stable set in $G_1$, contained in $V_1\setminus\{u\}$, and the set $I_2=I\cap V_2$ is stable in $G_2$. We consider the following two cases depending on whether the set $I_{1} \cup \{u\}$ is stable in $G_{1}$.

	\begin{enumerate}[label=(\roman*)]
		\item If ${I_{1} \cup \{u\}}$ is stable in $G_{1}$, then there exists a  nonnegative vector $t_{1} \in \R^{r_{1}}$ for which the vector $(y_{1},t_{1}) = (\chi^{I_{1} \cup \{u\}}, t_{1})$ satisfies the system $E_1y_1+F_1t_1=g_1$. Similarly, since $I_{2}$ is stable in $G_{2}$, there exists a nonnegative vector $t_{2}\in \R^{r_{2}}$ for which the vector $(y_{2},t_{2}) = (\chi^{I_{2}}, t_{2})$ satisfies the linear system $E_2y_2+F_2t_2=g_2$. As $y_{1}(u) = 1$, the vector $(y_{1}, y_{2}, t_{1}, t_{2})$ satisfies the linear system $(\ref{sublemma2sys3})$.

		\item If ${I_{1} \cup \{u\}}$ is not stable in $G_{1}$, then $u$ is adjacent to one vertex in $I_1$ and thus 
		$I_{2} = \emptyset$. 
		As $I_1$ is stable in $G_1$, there exists a nonnegative vector $t_{1}\in \R^{r_{1}}$ such that $(y_{1},t_{1}) = (\chi^{I_{1}}, t_{1})$ satisfies the system $E_1y_1+F_1t_1=g_1$. Note that $y_1(u)=0$ as $u\not\in I_1$. Taking $y_{2} =0$ and $t_{2} = 0$, we have that $(y_{2},t_{2})$ satisfies  $E_{2}y_{2}+F_{2}t_{2} - g_{2} \cdot y_{1}(u) = 0$.
		Thus  $(y_{1},y_{2},t_{1},t_{2})$ satisfies the linear system  $(\ref{sublemma2sys3})$. 
	\end{enumerate}
	In both cases,  we have constructed lifting variables $(t_{1},t_{2},y_{1}(u)) \in \R^{r_{1}+r_{2}+1}$ such that the vector $(y_{1},y_{2},t_{1},t_{2}) \in \R^{|V_{1}|+|V_{2}|+r_{1}+r_{2}}$ satisfies the linear system  $(\ref{sublemma2sys3})$. \\

	We now show the  ``if' part".  Assume  $(y_{1},y_{2},t_{1},t_{2}) \in \R^{|V_{1}|+|V_{2}|+r_{1}+r_{2}}$ satisfies   the linear system (\ref{sublemma2sys3}). 
Assume first $y_1(u)=0$. Then the conditions  $E_2y_2+F_2t_2=0,$ $t_2\ge 0$ imply $y_2=0$ (by Lemma \ref{sublemma1} applied to $\stab(G_2)$). Moreover the conditions $E_1y_1+F_1t_1=g_1,$ $t_1\ge 0$ imply that $y_1\in \stab(G_1)$. Hence $y_1$ is a convex combination of characteristic vectors of stable sets $I_1\subseteq V_1\setminus \{u\}$, which also gives a decomposition of the vector $(y_1(\ou),y_2)$ as a convex combination of characteristic vectors of stable sets in $G$.

	We may now assume $y_1(u)\ne 0$. As $(y_{1},y_{2},t_{1},t_{2})$ satisfies  the system (\ref{sublemma2sys3}) we deduce that $y_1\in \stab(G_1)$ and ${1\over y_1(u)}y_2\in \stab(G_2)$.
Say $$y_1=\sum_{I\in \MI_1}\lambda_{I}	\chi^I,\ \ y_2/y_1(u)=\sum_{J\in \MI_2} \mu_J\chi^J,$$
where all sets in $\MI_1$ (resp., $\MI_2$) are stable sets in $G_1$ (resp., $G_2$),  $\sum_I\lambda_I=\sum_J\mu_J=1$ and $\lambda_I,\mu_J>0$.
Then $y_1(u)=\sum_{I\in \MI_1: u\in I}\lambda_I$ and we have the identity
$$\sum_{I\in \MI_1:u\not\in I}\lambda_I \begin{pmatrix} \chi^I\cr 0\end{pmatrix}
+\sum_{I\in \MI_1: u\in I} \sum_{J\in \MI_2} \lambda_I\mu_J \begin{pmatrix} \chi^{I\setminus\{u\}}\cr \chi^J\end{pmatrix} =
\begin{pmatrix} y_1(\ou)\cr y_2\end{pmatrix}.$$ 
	All coefficients are nonnegative and their sum is
	$\sum_{I\in \MI_1:u\not\in I}\lambda_I +\sum_{I\in \MI_1: u\in I} \sum_{J\in \MI_2} \lambda_I\mu_J
	=1-y_1(u) + y_1(u)=1$.	
	Moreover, if $I$ is a stable set of $G_1$ with $u\in I$ and $J$ is a stable set of $G_2$, then the set $(I\setminus \{u\} )\cup J$ is stable in $G$.
So we have shown that the vector $(y_1(\ou),y_2)$ belongs to the stable set polytope of $G$.	
\end{proof}


\begin{remark}\label{rembetter}
One can    show a slightly tighter upper bound for $\xc(\stab(G))$ when $G$ is obtained by substituting  at a vertex of $G_1$ the graph  $G=K_p$ or $\overline {K_2}$.
Indeed, one can show that  $\xc(\stab(G)) \leq \xc(\stab(G_{1})) + p$ when $G_{2} = K_{p}$, and $\xc(\stab(G)) \leq \xc(\stab(G_{1})) + 3$ when $G_{2} = \overline{K_{2}}$, see \cite{thesisHao} for details.
 This is  a slight improvement over the result from Theorem \ref{subtheorem} which would, respectively, give  the bounds 
$\xc(\stab(G_1))+p+1$ and $\xc(\stab(G_1))+4$, using the fact that $\xc(\stab(K_p))=p+1$ and $\xc(\stab(\overline{K_2}))= 4$.
\end{remark}



We conclude with some applications of this bounding technique for graph substitution.

\begin{lemma}  \label{Uedgescomplement}
(i) If $G$ is the complete bipartite graph $K_{p,q}$ then $\xc(\stab(G))\le 2p+2q+3$.
 (ii) If $G$ is  the complement of the disjoint union of $p$ edges then
 $\xc(\stab(G)) \leq  4p+1.$ 
  \end{lemma}

\begin{proof}
(i) The complete bipartite graph  $G=K_{p,q}$ can be obtained by considering an edge $\{u,v\}$ for $G_1$ and successively substituting $\overline {K_p}$ at $u$ and $\overline {K_q}$ at $v$. Applying Theorem \ref{subtheorem} we obtain $\xc(\stab(G))\le \xc(\stab(\overline{K_p}))+\xc(\stab(\overline {K_q}))+\xc(\stab(K_2)) = 2p+2q+3.$\\
(ii) If $G$ is the complement of the union of $p$ edges  then $G$ can be obtained by successively substituting $\overline{K_2}$ at each vertex of the complete graph $K_p$. By Remark \ref{rembetter} we obtain that
$\xc(\stab(G))\le 	\xc(\stab(K_p))+ 3p= p+1+3p=4p+1.$
\end{proof}

Using Lemma \ref{Uedgescomplement}(ii) one can  sharpen the bound of Lemma \ref{lemdoublesplit} when $G$ is a double-split graph with parameters $(p,q,L_1,\cdots,L_p)$ and show 
$\xc(G)\le 5p+4q+3 \ (\le 5p+5q+2)$. 

As $K_{p,q}$ has $pq$ edges,  the bound from Lemma \ref{SGlemma3} is quadratic in the number of vertices while by Lemma \ref{Uedgescomplement}(i) the extension complexity of $\stab(K_{p,q})$ is linear in the number of vertices.

\subsection{2-Join decompositions}\label{CH_PP:2join}
Here we consider how the extension complexity of the stable set polytope behaves under  2-join decompositions.

\begin{theorem} \label{CH_PG:2jointheorem}
	Let $G$ be a perfect graph and let  $(V_{1},V_{2})$ be a partition of $V$ providing a 2-join decomposition  of $G$ as in Definition \ref{def2join}. Then we have 
	\[\xc(\stab(G)) \leq 3 \cdot \xc(\stab(G[V_{1}])) + 3\cdot \xc(\stab(G[V_{2}])).\]
\end{theorem}

\begin{proof}
 As $G$ is perfect we need to show
$\nrank (S_G)\le 3\cdot \nrank (S_{G_1})+3\cdot \nrank (S_{G_2})$. For this we examine the block structure of the slack matrix $S_G$ from (\ref{eqSG}).  
As we have no control on the size of the set $\MC_{12}$ of maximal mixed cliques, we examine in more detail how the mixed cliques and independent sets arise.
For $k=1,2$,  let $A_k,B_k$ be the subsets of $V_k$ as in Definition \ref{def2join} and set $D_k=V_k\setminus (A_k\cup B_k)$.  

Any mixed maximal clique is of the form $C=C_1\cup C_2$ where, either $C_1\subseteq A_1$ and $C_2\subseteq A_2$
(call $\MC_A$ the set of such maximal cliques), or $C_1\subseteq B_1$ and $C_2\subseteq B_2$ (call $\MC_B$ the set of such maximal cliques), so that 
$\MC_{12}=\MC_A\cup \MC_B$.
One can verify that $\MI_{12}=\MI_3\cup\MI_4\cup \MI_5\cup \MI_6$, where
$\MI_3$ (resp., $\MI_4$, $\MI_5$, $\MI_6$) contains the independent sets of the form $I\cup J$ with
$I\subseteq D_1$ and $J\subseteq V_2$ 
(resp., with $I\subseteq D_1\cup A_1$ and $J\subseteq D_2\cup B_2$, $I\subseteq D_1\cup B_1$ and $J\subseteq D_2\cup A_2$, 
$I\subseteq V_1$ and $J\subseteq D_2$). Recall that $\MMR_k=V_k\cup\MC_k$ for $k=1,2$.
With respect to these partitions of its row and column index sets the matrix $S_G$ has the block form:

\[S_{G}  = 
\bordermatrix{ & \mathcal{I}_{1} & \mathcal{I}_{2} & \mathcal{I}_{3} & \mathcal{I}_{4} & \mathcal{I}_{5} & \mathcal{I}_{6} \cr
\mathcal{R}_{1}	& S_{1,1} &	S_{1,2}  &	S_{1,3}	& S_{1,4} & S_{1,5}  & S_{1,6} \cr
\mathcal{R}_{2} & S_{2,1} & S_{2,2}  &  S_{2,3} & S_{2,4} & S_{2,5}  & S_{2,6} \cr
\mathcal{C}_{A}& S_{3,1} &	S_{3,2}	 &	S_{3,3}	& S_{3,4} & S_{3,5}  & S_{3,6} \cr
\mathcal{C}_{B}& S_{4,1} &	S_{4,2}	 &	S_{4,3}	& S_{4,4} & S_{4,5}  & S_{4,6} \cr }
\]

To conclude the proof it suffices to make the following observations.
For $k=1,2$, we have $S_{k,k}=S_{G_k}$, 	each column of $S_{k,3}, S_{k,4}, S_{k,5}, S_{k,6}$ is copy of a column of $S_{k,k}$, and each column of $S_{1,2}$ (resp., $S_{2,1}$) coincides with the column of $S_{1,1}$ (resp., $S_{2,2}$) indexed by the empty set.
Moreover, for $k=3,4$, $S_{k,1}$ is a submatrix of $S_{G_1}$, $S_{k,2}$ is a submatrix of $S_{G_2}$, and each column of $S_{k,3}, S_{k,4}, S_{k,5}, S_{k,6}$ is copy of a column of $S_{k,1}$ or $S_{k,2}$.
Combining these observations with Lemma \ref{ranklemmaBlock} gives the desired inequality.
\end{proof}
\subsection{Skew partitions}\label{CH_PP:skewpartition}

We examine now the behaviour of the extension complexity under skew partitions.

\begin{theorem} \label{CH_PG:skewpartitiontheorem}
	Let $G=(V,E)$ be a perfect graph and let $(A_{1},B_{1},A_{2},B_{2})$ be a  partition of $V$ providing a skew partition decomposition of $G$ as in Definition \ref{skew}.  Then we have
	\begin{align*}
	\xc(\stab(G)) & \leq  2 \cdot \xc(\stab(G[A_{1}\cup B_{1}])) + 2\cdot \xc(\stab(G[A_{2}\cup B_{2}])) \\
	& + \xc(\stab(G[A_{1}\cup B_{2}])) + \xc(\stab(G[A_{2}\cup B_{1}])).
	\end{align*}
\end{theorem}

\begin{proof}
It suffices to  show that $\nrank(S_G)$ is at most 
$$  2 \nrank(S_{G[A_1\cup B_1]}) +2 \nrank(S_{G[A_2\cup B_2]}) +\nrank(S_{G[A_1\cup B_2]}) +\nrank(S_{G[A_2\cup B_1]}).$$
For this we exploit the block structure of $S_G$ in (\ref{eqSG}), using the partition $V=V_1\cup V_2$ with  $V_k=A_k\cup B_k$ for $k=1,2$.
The mixed maximal cliques of $G$
are of the form
$C_1\cup C_2$,   either with $C_1\subseteq A_1\cup B_1$ and $C_2\subseteq A_2$ (call their set $\MC_3$), or 
with $C_1\subseteq A_1$ and $C_2\subseteq A_2\cup B_2$ (call their set $\MC_4$).
The mixed  independent sets of $G$ 
are of the form  $I_1\cup I_2$, either  with $I_1\subseteq A_1\cup B_1$ and $I_2\subseteq B_2$
(call their set $\MI_3$), or 
with  $I_1\subseteq B_1$ and $I_2\subseteq A_2\cup B_2$ (call their set $\MI_4$).
With respect to these partitions of its row and column index sets, the slack matrix $S_G$ has the block form:

\[S_{G}  = 
\bordermatrix{ 	& \mathcal{I}_{1} & \mathcal{I}_{2} & \mathcal{I}_{3} & \mathcal{I}_{4}  \cr
\mathcal{R}_{1}	& S_{1,1} &	S_{1,2}  &	S_{1,3}	& S_{1,4} \cr
\mathcal{R}_{2} & S_{2,1} & S_{2,2}  &  S_{2,3} & S_{2,4} \cr
\mathcal{C}_{3} & S_{3,1} &	S_{3,2}	 &	S_{3,3}	& S_{3,4} \cr
\mathcal{C}_{4} & S_{4,1} &	S_{4,2}	 &	S_{4,3}	& S_{4,4} \cr  }
\]

As in earlier proofs, we have $\nrank(S_{k,1}\ S_{k,2}\ S_{k,3} \ S_{k,4})  \le  \nrank(S_{G[A_k\cup B_k]})$ for $k=1,2$.
	Moreover, by looking at the shape of the mixed cliques and independent sets one can make the following observations: 
	$\nrank( S_{3,1}\ S_{3,3}) \le \nrank(S_{G[A_1\cup B_1]})$ since $S_{3,1}=S_{3,3}$ is a submatrix of $S_{G[A_1\cup B_1]}$,
	$\nrank(S_{3,2}\ S_{3,4})\le \nrank(S_{G[A_2\cup B_1]})$ since each column of $S_{3,2}$ is copy of a column of $S_{3,4}$ which in turn is a submatrix of $S_{G[A_2\cup B_1]}$,
	$\nrank(S_{4,2}\ S_{4,4})\le \nrank (S_{G[A_2\cup B_2]})$ since $S_{4,2}=S_{4,4}$ is a submatrix of $S_{G[A_2\cup B_2]}$,
	and 
	$\nrank(S_{4,1}\ S_{4,3})\le \nrank (S_{G[A_1\cup B_2]})$ since each column of $S_{4,1}$ is a copy of a column of $S_{4,3}$ which in turn is a submatrix of $S_{G[A_1\cup B_2]}$.
\end{proof}

\section{Application to perfect graphs}\label{secappli}

We now use the  above results  to  upper bound the extension complexity of the stable set polytope of a perfect graph $G$ by decomposing $G$  into basic perfect graphs using  2-join  and skew partition  decompositions.  Recall that basic perfect graphs are  bipartite graphs or their complements, line graphs of bipartite graphs or their complements, and  double-split graphs for which we know that $\xc(\stab(G))\le 2(|V|+|E|)$ (by Corollary \ref{corbasic}).

\begin{theorem}\label{CH_PG:2joinEX}
	Let  $G=(V,E)$ be  a perfect graph.
	Let $d$ be the depth of a decomposition tree representing a decomposition of $G$ into basic perfect graphs by means of 2-join and skew partition decompositions. Then we have
	\[ \xc(\stab(G)) \leq 4^{d}  (2|V|+2|E|).
	\]
\end{theorem}	
\begin{proof}
We use induction on the depth $d\ge 0$ of the decomposition tree. If $d=0$ then $G$ is a basic perfect graph and the result holds  by Corollary \ref{corbasic}.
Assume now $d\ge 1$, i.e., $G$ admits a 2-join or skew partition decomposition. We first consider the case when $G$ admits a 2-join decomposition $(V_1,V_2)$. Then, by Theorem \ref{CH_PG:2jointheorem}, we have
	$$\xc(\stab(G)) \leq 3 \cdot \xc(\stab(G[V_1])) + 3\cdot \xc(\stab(G[V_2])).$$ 
	By the induction assumption, we have 
	$\xc(\stab(G[V_k]))\le 4^{d-1}(2|V_k|+2|E_k|)$ for each $k=1,2$.
	As $|V|=|V_1|+|V_2|$ and $|E_1|+|E_2|\le |E|$,  we obtain the desired bound:
	$\xc(\stab(G))\le  3\cdot 4^{d-1}(2|V_1|+2|E_1|+2|V_2|+2|E_2|)\le 4^{d}(2|V|+2|E|)$.
	
	Consider now the case when $G$ admits a skew partition $(A_1,B_1,A_2,B_2)$.
We use Theorem $\ref{CH_PG:skewpartitiontheorem}$ which implies that 
	\[\xc(\stab(G)) \le 2 \cdot ( \xc(\hat{G}_1)+\xc(\hat{G}_2)+\xc(\hat{G}_3)+\xc(\hat{G}_3)),\]
	where $\hat{G}_1, \hat{G}_2,\hat{G}_3,\hat{G}_4$ are induced subgraphs of $G$ such that 
	$\sum_{k=1}^4 |V(\hat G_{k})|= 2|V|$ and $\sum_{k=1}^4|E(\hat G_k)|\le 2|E|$.
By the induction assumption,  for each $k=1,2,3,4$ we have: 
$\xc(\stab(\hat G_k))\le 4^{d-1} (2|V(\hat G_k)|+2|E(\hat G_k)|)$ .
Combining with the above relations we obtain the desired inequality:
$$\xc(\stab(G))\le 2\cdot 4^{d-1}\sum_{k=1}^4(2|V(\hat G_k)|+2|E(\hat G_k)|) \le  4^{d}  (2 |V|+2|E|).$$
\end{proof}

As $2|V|+2|E|\le 2|V|^2$, we derive the bound $\xc(\stab(G))\le 2\cdot 4^{d}|V|^2$ when $G$ has a decomposition tree of depth $d$. In particular, for the class of perfect graphs  $G$ admitting a decomposition tree whose depth $d$ is logarithmic in $|V|$, say
$d\le c\log |V|$ for some constant $c>0$,  the extension complexity of the stable set polytope is polynomial in $V$:
$$\xc(\stab(G))\le 2 |V|^{c+2}.$$

\medskip
To conclude let us remark that  other graph operations are known that preserve perfect graphs and  can be used to give structural characterizations for subclasses of perfect graphs.
This is the case in particular for the ``graph amalgam" operation considered in \cite{burlet1984polynomial}.
The behaviour of the amalgam operation is studied in \cite{conforti2013stable} (see also \cite{thesisHao}):
if $G$ is the amalgam of two perfect graphs $G_1$ and $G_2$ then $\xc(\stab(G))\le \xc(\stab(G_1))+\xc(\stab(G_2))$.
Burlet and Fonlupt \cite{burlet1984polynomial} introduce a notion of {\em basic} Meyniel graph and show that any 
 Meyniel graph can be decomposed  into {basic} Meyniel graphs using graph amalgams. It follows from results in Conforti et al.\,\cite{conforti2013stable} 
  that the extension complexity of the stable set polytope is polynomial  in the number of vertices for Meyniel graphs.
The question of deciding wether the extension complexity of the stable set polytope  is polynomial for all perfect graphs remains wide open.
  
  \bigskip
  {\bf Acknowledgments.}
  We thank Ronald de Wolf for useful discussions and comments about the topic of this paper.


\begin{thebibliography}{13}
\providecommand{\natexlab}[1]{#1}
\providecommand{\url}[1]{\texttt{#1}}
\expandafter\ifx\csname urlstyle\endcsname\relax
  \providecommand{\doi}[1]{doi: #1}\else
  \providecommand{\doi}{doi: \begingroup \urlstyle{rm}\Url}\fi

\bibitem{AFFHM}
M. Aprile, Y. Faenza, S. Fiorini, T. Huynh, M. Macchia.
Extension complexity of stable set polytopes of bipartite graphs.
arXiv:1702.08741v1, 2017.

\bibitem{Berge}
C. Berge.
F\"arbung von Graphen, deren s\"amtliche bzw. deren ungerade Kreise starr sind.
{\em Wiss. Z. Martin-Luther-Univ. Halle-Wittenberg Math.-Natur}, Reihe 114, 1961.


\bibitem[Burlet and Fonlupt(1984)]{burlet1984polynomial}
M.~Burlet and J.~Fonlupt.
\newblock Polynomial algorithm to recognize a {M}eyniel graph.
\newblock \emph{North-Holland mathematics studies}, 88:\penalty0 225--252,
  1984.

\bibitem{thesis}
M. Chudnovsky.
{\em  Berge Trigraphs and their Applications.}  Ph.D. thesis, Princeton University,
Princeton, NJ, 2003.

\bibitem{JGT}
M. Chudnovsky.
Berge trigraphs.
{\em Journal of Graph Theory}, 53(1), 1--55, 2006.

\bibitem[Chudnovsky et~al.(2006)Chudnovsky, Robertson, Seymour, and
  Thomas]{chudnovsky2006strong}
M. Chudnovsky, N. Robertson, P. Seymour, and R. Thomas.
\newblock The strong perfect graph theorem.
\newblock \emph{Annals of Mathematics}, 164(1):51--229, 2006.

\bibitem[Chv{\'a}tal(1975)]{chvatal1975certain}
V.  Chv{\'a}tal.
\newblock On certain polytopes associated with graphs.
\newblock \emph{Journal of Combinatorial Theory Series B}, 18\penalty0
  (2):\penalty0 138--154, 1975.

\bibitem[Chv{\'a}tal(1985)]{chvatal1985star}
V. Chv{\'a}tal.
\newblock Star-cutsets and perfect graphs.
\newblock \emph{Journal of Combinatorial Theory Series B}, 39\penalty0
  (3):\penalty0 189--199, 1985.

\bibitem{CCZ10}
M. Conforti, G. Cornu\'ejols, G. Zambelli.
Extended formulations in combinatorial optimization.
{\em 4OR}, 8:1--48, 2010.


\bibitem[Conforti et~al.(2013)Conforti, Gerards, and
  Pashkovich]{conforti2013stable}
M.  Conforti, B. Gerards, and K. Pashkovich.
\newblock Stable sets and graphs with no even holes.
{\em Mathematical Programming}, 153(1):13--39, 2015.



\bibitem[Cornu{\'e}jols(2001)]{cornuejols2001combinatorial}
G. Cornu{\'e}jols.
\newblock \emph{Combinatorial Optimization: Packing and Covering}, volume~74 of CBMS-NSF Regional Conference Series in Applied Mathematics, 
\newblock SIAM, 2001.


\bibitem{CC}
G. Cornu{\'e}jols and W.H. Cunningham.
Compositions for perfect graphs.
\emph{Discrete Mathematics,} 55:245--254, 1985.



\bibitem{FKPT}
S. Fiorini, V. Kaibel, K. Pashkovich, D.O. Theis.
Combinatorial bounds on negative rank and extended formulations.
{\em Discrete Mathematics,} 313(1):67--83, 2013.



\bibitem[Fiorini et~al.(2015)Fiorini, Massar, Pokutta, Tiwary, and
  Wolf]{fiorini2015exponential}
S. Fiorini, S. Massar, S. Pokutta, H.~R. Tiwary, and R.~de
  Wolf.
\newblock Exponential lower bounds for polytopes in combinatorial optimization.
\newblock \emph{Journal of the ACM}, 62\penalty0 (2),  Article No.17, 2015.

\bibitem{Gillis}
N. Gillis.
Introduction to nonnegative matrix factorization.
{\em Optima}, 25:7--16, 2017.

\bibitem{Goemans}
M. Goemans.
Smallest compact formulation for the permutahedron.
{\em Mathematical Programming}, 153(1):5--11, 2015.

\bibitem{GPT}
J. Gouveia, P. Parrilo and R. Thomas.
Lifts of convex sets and cone factorizations.
{\em Mathematics of Operations Research}, 38(2):248--264, 2013.

\bibitem[Gr{\"o}tschel et~al.(1981)Gr{\"o}tschel, Lov{\'a}sz, and
  Schrijver]{GLS}
M. Gr{\"o}tschel, L. Lov{\'a}sz, and A. Schrijver.
\newblock The ellipsoid method and its consequences in combinatorial
  optimization.
\newblock \emph{Combinatorica}, 1\penalty0 (2):\penalty0 169--197, 1981.

\bibitem{thesisHao}
 H. Hu. {\em On the extension complexity of stable set polytopes for perfect graphs.}
Master thesis, Utrecht University, 2015.
 
 \bibitem{Kaibel}
 V. Kaibel.
 Extended formulations in combinatorial optimization.
 {\em Optima}, 85:2--7, 2011.
 
\bibitem{LRS}
J.R. Lee, P. Raghavendra, D. Steurer.
Lower bounds on the size of semidefinite programming relaxations.
In {\em Proceedings of the 47th Annual ACM Symposium on Theory of Computing}, pages 567--576, 2015.



\bibitem{Lo72a}
L. Lov\'asz. Normal hypergraphs and the perfect graph conjecture, 
{\em Discrete Mathematics,} 2(3):253--267, 1972.

\bibitem{Lo72b}
L. Lov\'asz.
A characterization of perfect graphs.
{\em  Journal of Combinatorial Theory, Series B,} 13(2):95--98, 1972.

\bibitem{M}
R. Kipp Martin.
Using separation algorithms to generate mixed integer model reformulations.
{\em Operations Research Letters}, 10(3):119--128, 1991.

\bibitem[Meyniel(1976)]{meyniel1976perfect}
H. Meyniel.
\newblock On the perfect graph conjecture.
\newblock \emph{Discrete Mathematics}, 16\penalty0 (4):\penalty0 339--342,
  1976.


\bibitem{Rothvoss12}
T.  Rothvo{\ss}.
Some $0/1$ polytopes need exponential size extended formulations.
{\em Mathematical Programming, Ser. A}, 142:255--268, 2013.


\bibitem[Rothvo{\ss}(2014)]{rothvoss2014matching}
T. Rothvo{\ss}.
\newblock The matching polytope has exponential extension complexity.
\newblock In \emph{Proceedings of the 46th Annual ACM Symposium on Theory of
  Computing}, pages 263--272, 2014.


\bibitem{W}
R.T. Wong.
Integer programming formulations of the traveling salesman problem.
In {\em Proceedings of 1980 IEEE International Conference on Circuits and Computers}, pages 149--152, 1980.

\bibitem[Yannakakis(1988)]{yannakakis1988expressing}
M. Yannakakis.
\newblock Expressing combinatorial optimization problems by linear programs.
\newblock In \emph{Proceedings of the 20th Annual ACM Symposium on Theory
  of Computing}, pages 223--228, 1988.

\end{thebibliography}
\end{document}